\documentclass{lms}

\usepackage{xypic}
\usepackage{eucal}

\usepackage{amsmath}
\usepackage{amstext}
\usepackage{amssymb}
\usepackage{amscd}

\usepackage{lscape}
\usepackage{longtable}

\usepackage{epsfig}
\input txdtools
\let\et=\etexdraw
\def\etexdraw{\drawbb\et}

\newtheorem{thm}{Theorem}[section]
\newtheorem{lem}[thm]{Lemma}
\newtheorem{prop}[thm]{Proposition}
\newtheorem{cor}[thm]{Corollary}
\newtheorem{defi}[thm]{Definition}

\newnumbered{ex}[thm]{Example}
\newnumbered{qu}{Question}
\newnumbered{ntn}{Notation}
\newnumbered{rmd}{Reminder}
\newnumbered{rmk}{Remark}

\DeclareMathOperator{\height}{ht}
\DeclareMathOperator{\Hom}{Hom}
\DeclareMathOperator{\Ass}{Ass}

\DeclareMathOperator{\Ext}{Ext}

\DeclareMathOperator{\Ann}{ann}
\DeclareMathOperator{\Nil}{Nil}

\DeclareMathOperator{\grAnn}{gr-ann}
\DeclareMathOperator{\HH}{H}

\newcommand{\leftexp}[2]{{\vphantom{#2}}^{#1\!\!}{#2}}

\begin{document}

\title
[Frobenius maps on injective hulls]
{Frobenius maps on injective hulls and their applications to tight closure}

\author{Mordechai Katzman}

\classno{Primary 13A35, 13C11, 13D45, 13P99}



\maketitle

\begin{abstract}
This paper studies Frobenius maps on injective hulls of residue fields of complete local rings
with a view toward providing constructive descriptions of objects originating from the theory of
tight closure. Specifically, the paper describes algorithms for computing parameter test ideals, and
tight closure of certain submodules of the injective hull of residue fields of a class of well-behaved rings
which includes all quasi-Gorenstein complete local rings.
\end{abstract}

\section{Introduction}\label{Section: Introduction}

This paper studies problems originating from the theory of tight closure which we now review briefly.
Let $A$ be a commutative ring of prime characteristic $p$; for any positive integers $e$ we define the
\emph{iterated Frobenius endomorphism} $f^e: A \rightarrow A$
to be the map which raises elements to their $p^e$th power. This map can be used to endow $A$ with the structure of a $A$-bimodule.
As a left $A$-module it has the usual $A$-module structure whereas $A$ acts on itself on the right
via the iterated Frobenius map; we denote this bimodule $\leftexp{e}{A}$.
Now for all $a\in \leftexp{e}{A}$ and $b\in A$,
$b \cdot a=ba$ while $a\cdot b=b^{p^e} a$, where $\cdot$ denotes the action of $A$.
We can extend this construction to obtain the \emph{Frobenius functor} $F^e_A$
sending any $A$-module $M$ to $F^e_A(M)=\leftexp{e}{A}\otimes_A M$ where $A$ acts on $F^e_A(M)$ via its left-action on
$\leftexp{e}{A}$, so for $a\otimes m\in F^e_A(M)$ and $b\in A$
we have $b\cdot (a\otimes m)=ba\otimes m$ and $(a\otimes bm)=a \cdot b \otimes m=b^{p^e} a \otimes m$.

We often find it convenient to think of $\leftexp{e}{A}$ and the associated Frobenius functors as follows.
Let $\Theta$ be an indeterminate and consider the free $A$-module $A[\Theta; f^e]=\displaystyle \oplus_{i=0}^\infty A \Theta^{i}$
which we turn into a skew-polynomial ring by defining $\Theta a=a^{p^e} \Theta$ for all $a\in A$.
We can now identify $\leftexp{e j}{A}$ with $A \Theta^j \subset  A[\Theta; f^e]$ and for all
$A$-modules $M$ we may write $F^{e j}_A(M)=A \Theta^j \otimes_A M$.

If $M$ is an $A$-module and $N\subseteq M$ is an $A$-submodule we define the tight closure of $N$ in $M$, denoted $N^*_M$,
to be the set of all $m\in M$ such that for some $c\in A$ not in any minimal prime,
$c\otimes m \in F^e_A(M)$ is in the image of the map $F^e_A(N) \rightarrow F^e_A(M)$ for all $e\gg 0$.

Among the most interesting and useful results obtained early in the development of the theory of tight closure is the existence of \emph{test elements}
(cf.~Chapter 2 in \cite{Huneke}). Notice that the element $c\in A$ occurring in the definition of tight closure could depend on
the modules $N$ and $M$ and on the element $m\in M$.
Test elements are elements $c\in A$ not in any minimal prime such that for \emph{all} finitely generated modules $M$ and submodules $N\subseteq M$
and \emph{all} $m\in M$,
\begin{equation}\label{eqn: test elements}
m\in N^*_M \Leftrightarrow c \otimes m \in F^e_A(M) \text{ is in the image of }
F^e_A(N) \rightarrow F^e_A(M) \text{ for all } e\geq 0 .
\end{equation}
A weaker concept, that of a \emph{$p^{e^\prime}$-weak test element}
is defined similarly, only that we relax the last condition above and demand that
\begin{equation}\label{eqn: weak test elements}
m\in N^*_M \Leftrightarrow c \otimes m \in F^e_A(M) \text{ is in the image of }
F^e_A(N) \rightarrow F^e_A(M) \text{ for all } e\geq e^\prime .
\end{equation}
One also defines the
\emph{test ideal} and \emph{$p^{e^\prime}$-weak test ideal} of $A$ to be the ideals generated by all test elements,
and all $p^{e^\prime}$-weak test elements, respectively.

In many applications one restricts one's attention to local rings $A$ and to the tight-closure of ideals generated by systems of parameters.
One then naturally considers the notion of \emph{parameter test elements}: these are
elements $c\in A$ not in any minimal prime
which satisfy (\ref{eqn: test elements}) with $M=A$ and $N$ being an ideal
generated by a system of parameters.
Similarly one obtains the notion of \emph{$p^{e^\prime}$-weak parameter test elements}:
these are the elements $c\in A$ not in any minimal prime
which satisfy (\ref{eqn: weak test elements}) with $M=A$ and $N$ being an ideal
generated by a system of parameters.
One can then define
the
\emph{parameter test ideal} and \emph{$p^{e^\prime}$-weak parameter test ideal}) of $A$ to be the ideals generated by all parameter test elements,
and all $p^{e^\prime}$-weak parameter test elements, respectively.
It is worth noting that when $S$ is a Gorenstein ring, the notions of
parameter test ideals and test ideals coincide (cf.~Chapter 2 in \cite{Huneke}).

We refer the reader to the seminal paper \cite{Hochster-Huneke-0} and to \cite{Huneke}
for detailed descriptions of tight closure and its properties.

\bigskip
The main results of this paper produce explicit descriptions of these test ideals.
The first such result is Theorem \ref{Theorem: description of weak parameter test ideals} which gives a formula for
weak parameter test ideals of complete local rings. This is a generalization of Theorem 8.2 in \cite{Katzman} which gave a similar description of
the parameter test ideals of complete local rings under the assumption that a certain Frobenius map on
the the injective hull of the residue field is injective.

Another important result is Theorem \ref{Theorem: Formula for tight closure} which gives an explicit description of the tight closure
of certain submodules of the injective hull of the residue field of certain complete local rings. In view of the notorious difficulty
of computing the tight closure of ideals, the fact that sometimes it is easy to compute the tight closure of submodules of a much larger object seems
very interesting. Also, this result has immediate relevance to the study of test ideals.
It is known that test ideals of local rings are the annihilators of the finitistic
tight closure of $0$ in the injective hulls of their residue fields (cf.~section 8 of \cite{Hochster-Huneke-0})
and it is conjectured that this finitistic tight closure coincides with the regular tight closure (cf.~Conjecture 2.6 in \cite{Lyubeznik-Smith} and section 8 of that paper
where the conjecture is shown to hold in some cases.)
The last section of this paper computes the tight closure of $0$ in the injective hulls residue fields of certain complete local rings.

\bigskip
Throughout this paper, we fix $(R,\mathfrak{m})$ to be a complete regular ring of prime characteristic $p$,
we fix $I\subseteq R$ to be an ideal and we write $S=R/I$.
We denote with $E_R$ and $E_S=\Ann_{E_R} I$ the injective hulls of the residue fields of $R$ and $S$, respectively,
and $(-)^\vee$ shall denote the functor $\Hom_R(- , E)$.

\begin{defi}
For any $S$-module $M$ and all $e\geq 0$ we let $\mathcal{F}^e(M)$ denote the
set of all additive functions $\phi: M \rightarrow M$ with the property that $\phi(s m)=s^{p^e} \phi(m)$ for all
$s\in S$ and $m\in M$.
Note that each $\mathcal{F}^e(M)$ is naturally an $S$-module: for all $\phi\in \mathcal{F}^e(M)$ and $s\in S$ the map
$s\phi$ defined as $(s\phi)(m)=s\phi(m)$ for all $m\in M$ is in $\mathcal{F}^e(M)$.
We also define $\mathcal{F}(M)=\oplus_{e\geq 0}  \mathcal{F}^e(M)$.

We call an $S$-submodule $N\subseteq M$ an \emph{$\mathcal{F}^e(M)$-submodule} if
$\phi(N)\subseteq N$ for all $\phi\in \mathcal{F}^e(M)$;
if $N$ is an \emph{$\mathcal{F}^e(M)$-submodule} for all $e\geq 0$ we call $N$ an \emph{$\mathcal{F}(M)$-submodule.}
\end{defi}

We shall refer to the maps in $\mathcal{F}^e(M)$ defined above as \emph{$e$th Frobenius maps} (or just \emph{Frobenius maps} when $e=1$.)
The most important Frobenius map is, of course, \emph{the} Frobenius map on $f:S \rightarrow S$ given by $f(s)=s^p$.

Notice that given an $S$-module $M$, any $\phi\in \mathcal{F}^e(M)$ determines a left $S[\Theta; f^e]$-module structure
on $M$ given my $\Theta m=\phi(m)$ for all $m\in M$. Conversely, a left $S[\Theta; f^e]$-module structure on $M$
defines a $\phi\in \mathcal{F}^e(M)$ given by $\phi(m)=\Theta m$ for all $m\in M$.

We shall call an element $m$ of an $S[\Theta;f^e]$-module $M$ \emph{nilpotent} if $\Theta^{j} m=0$ for some $j\geq 0$ and we shall denote the set
all such elements $\Nil(M)$; this is easily seen to be an $S[\Theta;f^e]$-submodule of $M$.

In the first part of this paper we will be particularly interested in $S$-submodules of $E_S$
which are stable under one particular Frobenius map arising from a canonical Frobenius map  which we describe next.
One of most important examples of modules with Frobenius maps
is the top local cohomology module $\HH^d_{\mathfrak{m}S}(S)$ which is a left $S[T;f]$-module in the following natural way.
$\HH^d_{\mathfrak{m}S}(S)$ can be computed as the direct limit of
$$
\frac{S}{(x_1, \dots, x_d)S} \xrightarrow[]{x_1 \cdot \ldots \cdot x_d} \frac{S}{(x_1^2, \dots, x_d^2)S}  \xrightarrow[]{x_1 \cdot \ldots \cdot x_d} \dots
$$
where $x_1, \dots, x_d$ is a system of parameters of $S$ and
we can define a Frobenius map $\phi\in\mathcal{F}^e\left( \HH^d_{\mathfrak{m}S}(S)\right)$ on this direct limit by mapping the coset
$a + (x_1^n, \dots, x_d^n)S$ in the $n$-th component of the direct limit
to the coset $a^{p^e} + (x_1^{np^e}, \dots, x_d^{np^e})S$
in the $np^e$-th component of the direct limit.
When $S$ has a canonical module $\omega\subseteq S$, this $S[T;f]$-module structure induces one in $E_S$ as follows.
The inclusion $\omega\subseteq S$ yields a surjection $E_S=\HH^d_{\mathfrak{m}S}(\omega) \twoheadrightarrow \HH^d_{\mathfrak{m}S}(S)$
which can be made into a surjection of $S[T;f]$-modules by lifting the $S[T;f]$ module structure of $\HH^d_{\mathfrak{m}S}(S)$
onto $E_S$ (cf.~\S 7 in \cite{Katzman}).
It is this $S[T;f]$-module structure on $E_S$ which,
as in \cite{Katzman}, will  enable us to give a explicit description of the weak parameter test ideals of $S$.

Recall that as $R$ is a power series ring $\mathbb{K}[\![ x_1, \dots, x_n ]\!]$ for some field $\mathbb{K}$ of
characteristic $p$, $E_R$ is isomorphic to the module of inverse polynomials
$\mathbb{K}[ x_1^-, \dots, x_n^- ]$ (cf.~Example 12.4.1 in \cite{Brodmann-Sharp})
which has a natural left $R[T; f]$-module structure extending
$T x_1^{\alpha_1} \cdot \ldots \cdot x_n^{\alpha_n}=x_1^{p\alpha_1} \cdot \ldots \cdot x_n^{p\alpha_n}$ for all $\alpha_1, \dots, \alpha_n<0$.
One can show that all left $S[\Theta;f^e]$ module structures on $E_S=\Ann_{E_S} I$ are given by
$\Theta=u T^e$ where $u\in( I^{[p^e]} : I )$ (cf.~Proposition 4.1 in \cite{Katzman} and Chapter 3 of \cite{Blickle}.)
Given a left $S[T;f]$-module structure on $E_S$, the study of $S[T;f]$-submodules of $E_S$ now translates via Matlis duality to the study of certain
ideals of $R$:
\begin{defi}(cf.~Definition 4.2 in \cite{Katzman})
An ideal $J\subseteq S$ is called an $E_S$-ideal if
$\displaystyle \Ann_{E_S} J$
is an $S[T;f]$-submodule of $E_S$.
An ideal $J\subseteq R$ is called an $E_S$-ideal if it contains $I$ and its image in $S$ is an $E_S$-ideal.
\end{defi}
Theorem 4.3 in \cite{Katzman} states that an ideal $J\subseteq R$ containing $I$ is an $E_S$-ideal if and only if $u J \subseteq J^{[p]}$
where $u\in (I^{[p]} :_R I)$ determines the $S[T; f]$-module structure of $E_S$ as above.
It is this characterization which allows one to transform a question regarding submodules of $E_S$ to one regarding ideals of $R$,
and these transformations sometimes renders them tractable.

Notice that for an ideal $J\subseteq S$, being an $E_S$-ideal is equivalent to
$\displaystyle \Ann_{E_S} J= \Ann_{E_S} J S[T;f]$. We also note that when $S$ is Gorenstein the notion
of $E_S$-ideals coincides with that of $F$-ideals studied in \cite{Smith2}.

As in \cite{Katzman}
let $\mathcal{C}^e$ be the category of Artinian $S[T;f^e]$-modules and
let
$\mathcal{D}^e$ be the category of $R$-linear maps $M \rightarrow F^e_R(M)$ where $M$ is
a finitely generated $S$-module and where a morphism between $M\xrightarrow[]{a} F_R^e(M)$ and
$N\xrightarrow[]{b} F_R^e(N)$ is a commutative diagram of $R$-linear maps
\begin{equation*}
\xymatrix{
M \ar@{>}[d]^{a} \ar@{>}[r]^{\mu} & N \ar@{>}[d]^{b}\\
F_R^e(M) \ar@{>}[r]^{F^e_R(\mu)} & F_R^e(N)\\
} .
\end{equation*}
This paper uses mutually inverse functors
$\Delta^e: \mathcal{C}^e \rightarrow \mathcal{D}^e$ and
$\Psi^e: \mathcal{D}^e \rightarrow \mathcal{C}^e$
(originally introduced in \cite{Katzman}) which are defined as follows.
For $M\in \mathcal{C}^e$ we have an $R$-linear map
$\alpha_M: F^e_R(M) \rightarrow M$
given by $\alpha(r\otimes m)=r T m$ for all $r\in R$ and $m\in M$.
Applying $(-)^\vee$  to the map $\alpha$ one obtains an $R$-linear map
$\alpha_M^\vee: M^\vee \rightarrow  F_R^e(M)^\vee$.
We now use a functorial isomorphism
$\gamma_M: F_R^e(M)^\vee \rightarrow F^e_R(M^\vee)$ (cf.~Lemma 4.1 in \cite{Lyubeznik}) and
define  $\Delta(M)$ to be the map
$M^\vee \xrightarrow[]{\gamma_M \circ \alpha_M^\vee} F_R^e(M^\vee)$.
To define  $\Psi^e$ we retrace the steps above;
given a finitely generated $S$-module $N$ and a $R$-linear map
$a: N \rightarrow F^e_R(N)$
we define $\Psi^e(-)$ to coincide with the functor $(-)^\vee$ as a functor of $S$-modules
giving $\Psi^e(N)$ the additional structure of an $S[T;f^e]$-module structure as follows.
We apply ${}^\vee$ to the map $a$ above to obtain a map
$a^\vee : F_R^e(N)^\vee \rightarrow N^\vee$.
We next obtain a map $\epsilon: F_R^e\left(N^\vee\right) \rightarrow F_R^e(N)^\vee$ as the following composition:
$$F_R^e\left(N^\vee\right) \cong
F_R^e\left(N^\vee\right)^{\vee\vee} \xrightarrow[]{\left(\gamma_{N^\vee}^\vee\right)^{-1}}
F_R^e\left(N^{\vee\vee}\right)^{\vee} \cong
F_R^e\left(N\right)^{\vee}.$$
We now obtain a functorial map
$b=a^\vee\circ \epsilon: F^e_R(N^\vee) \rightarrow N^\vee$ and
we define the action of $T$ on $N^\vee$ by defining
$T n=b(1\otimes n)$  for all $n\in N^\vee$.

We shall use the functors $\Delta^e$ and  $\Psi^e$
(and in the proof of Theorem \ref{Theorem: graded annihilators} also details of the construction of $\Delta^e$),
to translate problems involving the injective hull $E_S$ to problems involving  ideals in $R$.
The crucial tool in tackling the latter will be the ideal operation $I_e(-)$: for an ideal $J\subseteq R$,
$I_e(J)$ is defined as the smallest ideal $L\subseteq R$ for which $J\subseteq L^{[p^e]}$.
The existence of this operation and its construction are discussed in section 5 of \cite{Katzman};
we shall assume the reader is familiar with the basic properties of this operation described there.

This paper is organized as follows:
Section \ref{Section: Basic properties} studies basic properties of submodules of $E_S$ and their annihilators which are used throughout this paper.
Section \ref{Section: Weak parameter test ideals}
generalizes Theorem 8.2 in \cite{Katzman} and gives an explicit description of the weak parameter test ideals
of $S$ in the case where $S$ is Cohen-Macaulay with canonical module $\omega\subseteq S$ but where the
Frobenius map on $E_S$ induced from the natural Frobenius map on $\HH^{\dim S}_{\mathfrak{m}S} (S)$ is not necessarily injective.
Section \ref{section: quasimaximal filtrations}
introduces a certain operation on $E_S$-ideal and applies it to the description
of quasimaximal filtrations of $E_S$. This operation is again used in section
\ref{Tight closure in $E_S$} which gives fairly explicit descriptions of the tight closure of certain
submodules of $E_S$.

\section{Basic properties of graded annihilators and $E_S$-ideals}
\label{Section: Basic properties}

Throughout this section we consider a fixed $S[T;f]$-module structure of $E_S$ corresponding to a fixed $u\in (I^{[p]} : I)$,
as described in section \ref{Section: Introduction}.

We start by listing some basic properties of $E_S$-ideals.

\begin{prop}\label{Proposition: basic properties of $E_R$-ideals}
\begin{enumerate}
\item[(a)] The intersection of $E_S$-ideals is an $E_S$-ideal.
\item[(b)] If $J\subseteq R$ is an $E_S$-ideal and $A\subset R$ is an ideal, then $(J : A)$ is an $E_S$-ideal.
\item[(c)] Assume $J\subseteq R$ is an $E_S$-ideal with minimal primary decomposition $Q_1 \cap \dots \cap Q_n$ and write $P_i=\sqrt{Q_i}$ for all
$1\leq i \leq n$.
Then $P_1, \dots, P_n$ are $E_S$-ideals and, if $P_i$ is not an embedded prime, then $Q_i$ is an $E_S$-ideal.
\end{enumerate}
\end{prop}
\begin{proof}
Let $\{ J_\lambda \}_{\lambda\in \Lambda}$ be a set of $E_S$-ideals. We have
$$u \left( \bigcap_{\lambda\in \Lambda} J_\lambda \right) \subseteq
\bigcap_{\lambda\in \Lambda} u J_\lambda \subseteq
\bigcap_{\lambda\in \Lambda} J_\lambda^{[p]} =
 \left( \bigcap_{\lambda\in \Lambda} J_\lambda \right)^{[p]} $$
where the equality follows from the fact that $R^{1/p}$ is an $\cap$-flat $R$-module (cf. Proposition 5.3 in \cite{Katzman})
and (a) follows.

Since
$$ u (J : A) A^{[p]} \subseteq  u (J : A) A \subseteq u J \subseteq J^{[p]}$$
we see that
$$ u (J : A) \subseteq (J^{[p]} :  A^{[p]}) = (J :  A)^{[p]} $$
where the equality follows from the fact that $R$ is regular, and now (b) follows.

To prove (c), first assume that $P_i$ is not an embedded prime, and pick $a\in \cap_{j\neq i} Q_j \setminus P_i$.
Now
$$(J : a) = \bigcap_{j=1}^n (Q_j : a) = (Q_i : a) = Q_i$$
is an $E_S$-ideal.
Any $P_i$ has the form $(J : a)$ for some $a\in R$, so (b) implies that $P_i$ is an $E_S$-ideal.
\end{proof}

\begin{defi}
Let $H$ be an $S[T; f]$-module and let $M\subseteq H$ be an $S$-submodule.
For any $e\geq 0$ we write $S T^e M$ for the $S$-submodule of $M$ generated by
$\{ T^e m \,|\, m\in M \}$ and we also write $M^{(e)} =( 0 :_R ST^e M)$.
We define the \emph{graded annihilator of $M$},
denoted $\grAnn M$, to be the ideal $\oplus_{e\geq 0} M^{(e)}S T^e\subseteq S[T; f]$.

We shall call an ideal $L\subseteq S$ \emph{$H$-special,} if there exists an $S[T; f]$-submodule $N\subseteq H$ for which $\grAnn N=L S[T; f]$.
When $H=E_S$ and $\Nil(E_S)=0$ the notions of $E_S$-special ideals and $E_S$-ideals coincide (cf.~\S 6 in \cite{Katzman}).
\end{defi}

Note that whenever $M\subseteq H$ is an $S[T; f]$-submodule,
$\left\{ M^{(e)}\right\}_{e\geq 0}$ is an ascending chain of ideals.
When $\Nil(M)=0$ that ascending chain is constant
and that constant value is a radical $M$-special ideal,
whose minimal primes are themselves $M$-special ideals. (cf.~Corollary 3.7 in \cite{Sharp}).
In general the ascending chain  $\left\{ M^{(e)}\right\}_{e\geq 0}$
need not be constant (e.g., while $\Nil(E_S)^{(e)}=S$ for all large $e$, $\Nil(E_S)^{(0)}\neq S$
whenever $\Nil(E_S)\neq 0$), and the ideals there may be non-radical. We next study the properties of these chains of ideals.

\begin{lem}\label{Lemma: dual of quotient}
Let $J_1\subseteq J_2\subseteq R$ be any ideals.
$$\left( \frac{J_2}{J_1} \right)^\vee \cong \frac{\Ann_{E_R} J_1}{\Ann_{E_R} J_2}  .$$
\end{lem}
\begin{proof}
Apply $(-)^\vee$ to the short exact sequence
$$0 \rightarrow J_2/J_1 \rightarrow R/J_1 \rightarrow R/J_2 \rightarrow 0$$
to obtain
the short exact sequence
$$0 \rightarrow {\Ann_{E_R} J_2} \rightarrow {\Ann_{E_R} J_1} \rightarrow \left( \frac{J_2}{J_1} \right)^\vee \rightarrow 0 .$$
\end{proof}

For any $e\geq 1$ write $\nu_e=1+\dots+p^{e-1}$.

\begin{thm}\label{Theorem: graded annihilators}
Let $M$ be an $S$-submodule of $E_S$ and write $M=\Ann_{E_S} L$ for some ideal $L\subseteq R$.
\begin{enumerate}
  \item[(a)] For all $e\geq 0$, $$ST^e M  \cong \frac{ \Ann_{E_S} L^{[p^e]}}{\Ann_{E_S} (u^{\nu_e} R + L^{[p^e]})} .$$
  \item[(b)] For all $e\geq 0$, $M^{(e)}=(L^{[p^e]} : u^{\nu_e})$.
  \item[(c)] For all $e\geq 0$, $u M^{(e)} \subseteq {M^{(e-1)}}^{[p]}$.
  \item[(d)] Assume further that $M$ is an $S[T;f]$-submodule of $E_S$. Then $\Ann_{E_S} M^{(e)}$ is an $S[T; f]$-submodule of $E_S$ and
  if for some $e\geq 0$ we have $M^{(e)}=M^{(e+1)}$, then $M^{(j)}=M^{(e)}$ for all $j\geq e$.
\end{enumerate}
\end{thm}
\begin{proof}
Fix any $e\geq 0$ and consider the map of $R$-modules $\psi_e : R T^e \otimes_R E_S \rightarrow E_S$ given by
$\psi_e(r T^e\otimes m)=r T^e m$; notice that  $\psi_e(R T^e \otimes_R M)=S T^e M$.
We also note that $\psi_e$ is the composition
$F^{e-1}_R(\psi_1) \circ \dots \circ F^{1}_R(\psi_1) \circ \psi_1$ and that the Matlis dual of this map is then given by
multiplication by $u u^{p} \cdot \ldots \cdot u^{p^{e-1}}=u^{\nu_e}$ (cf.~Proposition 4.5 in \cite{Katzman}.)

Since $R$ is regular, we have an injection
$R T^e \otimes_R M \hookrightarrow R T^e \otimes_R E_S$; let $\overline{\psi}_e$ be the restriction of
$\psi_e$ to $R T^e \otimes_R M$ and consider
the following commutative diagram
\begin{equation*}\label{CD5}
\xymatrix{
R T^e \otimes E_S \ar@{>}[rr]^{\psi_e} & & E_S \ar@{=}[d]\\
R T^e \otimes M \ar@{_{(}->}[u] \ar@{>>}[r]^{\overline{\psi}_e} & ST^e M \ar@{^{(}->}[r]& E_S \\
}.
\end{equation*}
An application of Matlis duality
together with the fact that
$ \left(R T^e  \otimes_R M \right)^\vee \cong R T^e  \otimes_R M^\vee$ (cf.~Lemma 4.1 in \cite{Lyubeznik})
yields the commutative diagram
\begin{equation*}\label{CD6}
\xymatrix{
R/I^{[p^e]}  \ar@{>>}[d] & & R/I \ar@{>}[ll]_{u^{\nu_e}} \ar@{=}[d]\\
R/L^{[p^e]}   & (ST^e M)^\vee \ar@{_{(}->}[l]& R/I \ar@{>>}[l]\\
} .
\end{equation*}
The image of the composition of the top and left maps is $u^{\nu_e}R + L^{[p^e]}/L^{[p^e]}$ and this coincides with the image of
$(ST^e M)^\vee$ in $R/L^{[p^e]}$.
We deduce that $( ST^e M )^\vee$ is isomorphic  to $u^{\nu_e} R + L^{[p^e]}/ L^{[p^e]}$.
Now
$$(ST^e M)\cong (ST^e M)^{\vee\vee}= \left(\frac{u^{\nu_e} R + L^{[p^e]}}{L^{[p^e]}}\right)^\vee$$
and an application of Lemma \ref{Lemma: dual of quotient} gives (a).

We now compute
\begin{eqnarray*}
M^{(e)} & = & ( 0 :_R ST^e M ) \\
    & = & ( 0 :_R \left(ST^e M\right)^\vee ) \\
    & = & \left( 0 :_R \frac{u^{\nu_e} R + L^{[p^e]}}{ L^{[p^e]} } \right)\\
    & = & ( L^{[p^e]} :_R u^{\nu_e} R )\\
\end{eqnarray*}
and obtain (b). Next we notice that, for all $e\geq 0$, $\nu_e-1=p\nu_{e-1}$ and
$$u M^{(e)}   =  u (L^{[p^e]} : u^{\nu_e})
     \subseteq  (L^{[p^e]} : u^{p \nu_{e-1}})
     =  (L^{[p^{e-1}]} : u^{\nu_{e-1}})^{[p]}
     =   {M^{(e-1)}}^{[p]}$$
and (c) follows.

If $M$  is an $S[T;f]$-submodule of $E_S$ then $\left\{ M^{(e)}\right\}_{e\geq 0}$ is an ascending chain of ideals
and we deduce from (c) that $u M^{(e)} \subseteq  {M^{(e)}}^{[p]}$, i.e., that
$M^{(e)}$ is an $E_S$-ideal and hence $\Ann_{E_S} M^{(e)}$ is an $S[T; f]$-submodule of $E_S$.

Consider the maps $\beta_i: R/L \rightarrow F_R^i(R/L)=R/L^{[p^i]}$ given by the composition
$$R/L \xrightarrow[]{u} R/L^{[p]} \xrightarrow[]{u^p} \dots   \xrightarrow[]{u^{p^{i-1}}}  R/L^{[p^i]} ,$$
i.e., by multiplication by $u^{\nu_i}$. For each $i\geq 1$, the kernel of $\beta_i$ is the image of
$M^{(i)}=( L^{[p^i]} :_R u^{\nu_i}  )$ in $R/L$; (d) now follows from Proposition 2.3(b) in \cite{Lyubeznik}.
\end{proof}

We can now prove the following generalization of Proposition 3.3 in \cite{Sharp}.

\begin{thm}\label{Theorem: graded annihilators of quotients}
Let $A\subseteq B$ be $S[T; f]$-submodules of $E_S$.
Write $A=\Ann_{E_S} K$ and $B=\Ann_{E_S} J$ for some ideals $J\subseteq K\subseteq R$.
Also write
$\grAnn B=\oplus_{e\geq 0} b_e T^e$ and
$\grAnn B/A=\oplus_{e\geq 0} \overline{b}_e T^e$
where $b_e$ and $\overline{b}_e$ are ideals of $R$ for all $e\geq 0$.
For all $e\geq 0$,
$$\overline{b}_e=\left( (J^{[p^e]} :_R u_{\nu_e}) :_R K \right)= (b_e :_R K) .$$
\end{thm}
\begin{proof}
As in the proof of Theorem \ref{Theorem: graded annihilators}, fix any $e\geq 0$ and consider the
map of $R$-modules $\psi_e : R T^e \otimes_R B/A \rightarrow B/A$ given by
$\psi_e(r T^e\otimes m)=r T^e m$; we notice that the image of $\psi_e$ is $(S T^e B+A)/A$.
An application of Matlis duality to the maps
$$  R T^e \otimes_R B/A \twoheadrightarrow (S T^e B+A)/A \hookrightarrow B/A$$
yields the $R$-linear maps
$$K/J \twoheadrightarrow ( (S T^e B+A)/A )^\vee \hookrightarrow K^{[p^e]}/J^{[p^e]}$$
whose composition is given by multiplication by $u^{\nu_e}$.
We deduce that $( (S T^e B+A)/A )^\vee$ is isomorphic to the image of the map $R/J \rightarrow R/J^{[p^e]}$ given by multiplication
by $u^{\nu_e}$, i.e., to $u^{\nu_e} K + J^{[p^e]}/ J^{[p^e]}$.
Now
\begin{eqnarray*}
\overline{b}_e &=&(0 :_R  (S T^e B+A)/A ) \\
&=& \left(0 :_R ( (S T^e B+A)/A )^\vee\right)\\
&=&( J^{[p^e]} :_R u^{\nu_e} K ) \\
&=&\left( (J^{[p^e]} :_R u^{\nu_e} ) :_R  K \right)\\
&=&(b_e :_R K) .
\end{eqnarray*}
\end{proof}

Recall that there exists an integer $\eta\geq 0$ such that $T^\eta \Nil(E_S) =0$ (cf.~Proposition 4.4. in \cite{Lyubeznik})
which we shall refer to as the \emph{index of nilpotency of $E_S$}
and that $\Nil(E_S)=\Ann_{E_S} I_\eta( u^{\nu_\eta} R) + I$
where for all ideals $L\subseteq R$ and positive integers $e$, $I_e(L)$ is defined as the smallest ideal $J$ for which $L\subseteq J^{[p^e]}$
(cf.~Theorem 4.6 and section 5 in \cite{Katzman}.)
\begin{cor}
Let $B$ be an $S[T; f]$-submodule of $E_S$ and write $B=\Ann_{E_S} J$ for some ideal $J\subseteq R$.
Let $\eta$ be the index of nilpotency of $E_S$  and $K=I_{\eta} (u^{\nu_\eta}R) + I$.
We have  $\left( (J^{[p^e]} :_R u^{\nu_e} ) :_R  K  \right)=
\left( (J :_R u ) :_R  K \right)$ for all $e\geq 0$ and these are radical ideals.
\end{cor}

We conclude this section by recording some additional properties of the associated primes of the ideals occurring in graded annihilators of
$S[T; f]$-submodules of $E_S$. These properties will not be used elsewhere in this paper.
\begin{prop}\label{Theorem: positive height}
Let $M$ be an $S[T; f]$-submodule of $E_S$, write $\overline{M}=M+\Nil(E_S)/\Nil(E_S)$ and write
$\Nil(E_S)=\Ann_{E_S} K$.
\begin{enumerate}
  \item[(a)] For all $e\geq 0$, $\Ass M^{(e)}\supseteq \Ass M^{(e+1)}$.
  \item[(b)] For all $e\geq 0$, if $P\in \Ass M^{(0)} \setminus \Ass M^{(e)}$ and $P$ is not an embedded prime of $M^{(0)}$
  then $P\supseteq K$.
  \item[(c)] Assume that $\height KS>0$. If $\height M^{(e)}S>0$ for some $e\geq 0$ then $\height M^{(e)}S>0$ for all
  $e\geq 0$.
\end{enumerate}
\end{prop}
\begin{proof}
Write $M=\Ann_{E_S} J$ for some $E_S$-ideal $J$.
For all $e\geq 0$ let $Q_1^{(e)} \cap \dots \cap Q_{n_e}^{(e)}$ be a minimal primary decomposition of $M^{(e)}$ and write
$P_i^{(e)}=\sqrt{Q_1^{(e)}}$ for all $1\leq i\leq n_e$.

Theorem \ref{Theorem: graded annihilators}(b) shows that
$$M^{(e+1)}=\left( J^{[p^{e+1}]} :_R u^{p \nu_e + 1} \right) = \left( \left( J^{[p^{e}]} :_R u^{\nu_e} \right)^{[p]} :_R u \right)=
\left( {M^{(e)}}^{[p]} :_R u \right) .$$
Now
$$ \bigcap \left\{ \left({Q_i^{(e)}}^{[p]} :_R u \right) \,|\, 1\leq i \leq n_e ,\ u\notin {Q_i^{(e)}}^{[p]} \right\}$$
is a primary decomposition of $M^{(e+1)}$ and (a) follows.

Theorem \ref{Theorem: graded annihilators of quotients} implies that $\overline{M}^{(e)}=(M^{(e)} :_R K)$ for all $e\geq 0$
and  $\overline{M}^{(e)}= \overline{M}^{(0)}$ is a radical ideal (cf.~Lemma 1.9 in \cite{Sharp}).
Now we obtain primary decompositions
$$ \bigcap \left\{ \left(Q_i^{(e)} :_R K\right) \,|\, 1\leq i \leq n_e ,\ K \nsubseteqq Q_i^{(e)} \right\}$$
where the primary components associated with minimal primes are irredundant,
and since these primary components occur for all $e\geq 0$, (b) follows.

Assume now that that $\height KS>0$ and that $\height M^{(e)}S>0$ for some $e\geq 0$.
If $\height M^{(0)}S=0$ then there exists an associated prime $P$ of $M^{(0)}$ such that $\height PS=0$.
But $P$ is not an associated prime of $M^{(e)}$, so $K\subseteq P$ and
$\height KS \leq \height PS =0$, a contradiction.
\end{proof}

All the results in this section have natural analogues when working
with a $S[\Theta; f^e]$-module structure on $E_S$-- these were omitted for the sake of simplicity. The proofs of this analogous results consist of straightforward modifications of the proofs
given above. In what follows we shall assume the more general results.

\section{Weak parameter test ideals}\label{Section: Weak parameter test ideals}
In this section we describe an algorithm for computing the $p^{e}$-weak parameter test ideal
of complete local Cohen-Macaulay rings. This extends the main result in \cite{Katzman} where this was done under the assumption that
a certain Frobenius map on $E_S$ is injective.

Throughout this section we assume $S$ to be Cohen-Macaulay with canonical module $\omega\subseteq S$ and we write $H=\HH^{\dim S}_{\mathfrak{m}S}(S)$.
Now $\dim S/\omega < \dim S$ and the short exact sequence
$$ 0 \rightarrow \omega \rightarrow S \rightarrow S/\omega \rightarrow 0 $$
yields a surjection $\Upsilon: E_S \twoheadrightarrow H$. We can now
endow $E_S$ with a structure of an $S[T;f]$-module
which makes this surjection into a map of $S[T; f]$-modules
(cf.~section 7 in \cite{Katzman}). We fix this $S[T;f]$-module structure throughout this section.

\bigskip
Let $J\subseteq R$ be henceforth in this section the ideal for which
$\ker \Upsilon = \Ann_{E_S} J$.
This ideal can be computed effectively as follows.
The map $\Upsilon$ is obtained from the long exact sequence of local cohomology modules arising from the the short exact sequence
$0 \rightarrow \omega \rightarrow S \rightarrow S/\omega\rightarrow 0$, i.e., from
$$ 0 \rightarrow \HH^{\dim S-1}_{\mathfrak{m}S}(S/\omega) \rightarrow \HH^{\dim S}_{\mathfrak{m}S}(\omega)  \xrightarrow[]{\Upsilon}
\HH^{\dim S}_{\mathfrak{m}S} (S) \rightarrow 0 .$$
We may rewrite this short exact sequence using local duality to obtain
$$ 0 \rightarrow \Ext^{\dim R - \dim S +1}_R (S/\omega,R)^\vee \rightarrow \Ext^{\dim R - \dim S}_R (\omega,R)^\vee  \xrightarrow[]{\Upsilon}
\Ext^{\dim R - \dim S }_R (S,R)^\vee \rightarrow 0 $$
which yields
$$ 0 \rightarrow \Ext^{\dim R - \dim S }_R (S,R) \rightarrow \Ext^{\dim R - \dim S}_R (\omega,R) \rightarrow (\ker \Upsilon)^\vee \rightarrow 0$$
Now $\Ext^{\dim R - \dim S}_R (\omega,R)\cong S$ and we identify
$\omega^\prime=\Ext^{\dim R - \dim S }_R (S,R)$, which is a canonical module for $S$, with its image in $S$.
We have
$(\ker \Upsilon)^\vee \cong S/\omega^\prime$ and
another application of $(-)^\vee$ gives
$(S/JS)^\vee \cong \ker \Upsilon\cong (S/\omega^\prime)^\vee$ and so
$S/JS\cong S/\omega^\prime$, and, therefore,
$JS=\omega^\prime$.


\begin{defi}
For all $e\geq 0$ we define
$$\mathcal{I}_e =\left\{ M^{(e)} \,|\, M\subseteq H \text{ is an } S[T; f]\text{-submodule} \right\} .$$
Notice that this extends the definition of the set of $H$-special ideals given in \cite{Sharp} for the case
where $H$ is $T$-torsion-free.
\end{defi}

Fix a system of parameters $x_1, \dots, x_d$ of $S$ and
think of $H$ as the direct limit
$$
\frac{S}{(x_1, \dots, x_d)S} \xrightarrow[]{x_1 \cdot \ldots \cdot x_d} \frac{S}{(x_1^2, \dots, x_d^2)S}  \xrightarrow[]{x_1 \cdot \ldots \cdot x_d} \dots .
$$
with its standard Frobenius described in the introduction and notice that as we assume $S$ to be Cohen-Macaulay,
the maps in this direct limit are injective.
Pick some element $a + (x_1^i, \dots, x_d^i)S$.
In what follows we will tacitly use the fact that
$\overline{c} T^e(a + (x_1^i, \dots, x_d^i)S)=0$ in the direct limit for some $\overline{c}\in S$ not in any minimal prime if and only if
$a\in ((x_1^i, \dots, x_d^i)S)^*$ (cf.~Remark 4.2 in \cite{Sharp}).

\begin{thm}
Assume that $S$ has a parameter test element.
For all $e\geq 0$, the $p^e$-weak parameter test ideal of $S$ is the image of
$$\cap \left\{ K \,|\, K\in \mathcal{I}_e, \height KS>0 \right\} $$
in $S$.
\end{thm}
\begin{proof}
Let $\tau$ be the intersection in the statement of the theorem.
Assume that $d$ is a $p^e$-weak parameter test element. If $M\subseteq H$ is an $S[T; f]$-submodule for which
$\height M^{(e)}S>0$, we can find a $c\in  M^{(e)}$ whose image in $S$ is not in a minimal prime such that
$c T^{e^\prime} M=0$ for all ${e^\prime}\geq e$, and hence $d T^{e^\prime} M=0$ for all ${e^\prime}\geq e$, and in particular $d\in M^{(e)}$.
We deduce that $d\in \tau$.
We next show that all elements in $\tau$ are  $p^{e}$-weak parameter test elements.

Fix a $c\in R$ whose image $\overline{c}$ in $S$ is a parameter test element.
Let $h\in H$ be such that such that $c T^{e^\prime} h=0$  for all ${e^\prime}\geq 0$.
Define $L=\oplus_{{e^\prime}\geq 0} S \overline{c} T^{e^\prime}$ and $M=\Ann_H L$; notice that $h\in M$.
Now $c\in M^{(0)}\subseteq  M^{(e)}$ and so $\height M^{(e)}S>0$.
Also $\tau\subseteq M^{(e)}$
so $\tau T^e M \subseteq  M^{(e)} T^e M = 0$, and in particular $\tau T^e h=0$.
\end{proof}

\begin{lem}\label{Lemma: positive height}
Assume that $S$ has a parameter test element.
Let $M$ be a $S[T; f]$-submodule of $H$.
If $\height M^{(e)} S>0$ for some $e\geq 0$, then $\height M^{(e^\prime)} S >0$ for all $e^\prime\geq 0$.
\end{lem}
\begin{proof}
We  assume that $\height M^{(e)} S >0$ for some $e\geq 0$ and show that $\height M^{(0)} S>0$. Since $ M^{(e^\prime)}\supseteq   M^{(0)}$
for all $e^\prime\geq 0$, we will then have  $\height M^{(e^\prime)} S >0 $ for all $e^\prime\geq 0$.

Pick any element $d\in  M^{(e)}$ whose image in $S$ is not in any minimal prime and notice that
$d\in  M^{(j)}$ for all $j\geq e$, i.e.,
$d ST^j M =0$ for all $j\geq e$.

Let $\mathbf{x}=(x_1, \dots, x_{\dim S})$ be a full system of parameters of $S$
and write $\mathbf{x}^nS $ for the ideal of $S$ generated by $x_1^n, \dots, x_{\dim S}^n$.
Now think of $H$
as the direct limit of
$$\frac{S}{\mathbf{x}S} \rightarrow \frac{S}{\mathbf{x}^2 S} \rightarrow \frac{S}{\mathbf{x}^3 S} \rightarrow \dots $$
where the (injective) maps are given by multiplication by $x_1 \cdot  \ldots \cdot x_{\dim S}$.

Any element $m\in M$ can be identified with an element represented by $s+\mathbf{x}^i S$ in the direct limit system above,
and the fact that $d ST^j m=0$ for all $j\geq e$ shows that $d s^{[p^j]} \in (\mathbf{x}^i S)^{[p^j]}$ for all $j\geq e$
and hence $s\in (\mathbf{x}^i S)^*$. Now for any parameter test element $c$,  $c (\mathbf{x}^i S)^* \subseteq \mathbf{x}^i S$, and we deduce that
$c s \in \mathbf{x}^i S$. We now see that any parameter test element kills $M$ and so is in $M^{(0)}S$, hence
$M^{(0)}S$ has positive height.

\end{proof}

We are now ready to give an explicit description of weak parameter test ideals and to do so
we need to recall the following notion (cf.~section 5 in \cite{Katzman}).
For any ideal $L\subseteq R$ and $u\in R$ we define  $L^{\star u}$ to be the smallest ideal
$A$ containing $L$ with the property that $u A \subseteq A^{[p]}$.

\begin{thm}\label{Theorem: description of weak parameter test ideals}
Let $c\in R$ be such that its image in $S$ is a test element.
For all $e\geq  0$, the $p^{e}$-weak parameter test ideal $\overline{\tau}_e$ of $S$ is given by
$$\left( \left(
\left(cJ+I\right)^{\star u} \right)^{[p^e]} :_R u^{\nu_e} J
\right)  S.$$
\end{thm}
\begin{proof}
Write $L=\left(cJ+I\right)^{\star u}$.
Notice that $L$ is an $E_S$-ideal and that since $c\in ((cJ+I)^{\star u} : J)$, we have $\height LS >0$.
Now
\begin{eqnarray*}
\overline{\tau}_e & = & \cap \{(M^{(e)}S \,|\, M\subseteq H \text{ is an } S[T;f] \text{-submodule},\ \height M^{(e)}S > 0 \}\\
&=&   \cap \{\left(A^{[p^e]} :_R u^{\nu_e} J \right) S \,|\,
A\subseteq J \text{ is an $E_S$-ideal},\ \height \left(A^{[p^e]} :_R u^{\nu_e} J \right) S > 0 \}
\end{eqnarray*}
and
$\left(L^{[p^e]} :_R u^{\nu_e} J \right) S$ is one of the ideals in this intersection,
hence $\overline{\tau}_e \subseteq \left(L^{[p^e]} :_R u^{\nu_e} J \right) S$.

Now let $A\subseteq J$ be any $E_S$-ideal for which $\height \left(A^{[p^e]} :_R u^{\nu_e} J \right) S > 0$.
Lemma \ref{Lemma: positive height} implies that $\height \left(\Ann_{E_S} A\right)^{(0)}=\height AS>0$ and
since the image of $c$ in $S$ is in $\overline{\tau}_0\subseteq (A : J)S$, we have $ c J \subseteq A$.
Proposition 5.5 in \cite{Katzman} now implies that $L \subseteq A$ and hence that
$  \left(L^{[p^e]} :_R u^{\nu_e} J\right) \subseteq  \left(A^{[p^e]} :_R u^{\nu_e}J \right) $.
We conclude that
$\left(L^{[p^e]} :_R u^{\nu_e} J \right) S \subseteq \overline{\tau}_e$.
\end{proof}

\begin{cor}\label{Corollary: stabilization of test ideals}
Let $\overline{\tau}$ be the union of the ascending chain $\{ \overline{\tau}_e \,|\, e\geq 0 \}$.
If $\overline{\tau}_i=\overline{\tau}_{i+1}$ for some $i\geq 0$ then $\overline{\tau}=\overline{\tau}_i$.
\end{cor}
\begin{proof}
Write $L=\left(cJ+I\right)^{\star u}$ and let $M=\Ann_{E_S} L$.
Theorem \ref{Theorem: description of weak parameter test ideals} together with Theorem \ref{Theorem: graded annihilators}(b)
imply that
$\grAnn M=\oplus_{e\geq 0} \overline{\tau}_e T^e$
and the result follows from Theorem \ref{Theorem: graded annihilators}(d).
\end{proof}

We can translate Theorem \ref{Theorem: description of weak parameter test ideals}
above to an algorithm as follows.
\begin{enumerate}
  \item Given $R$ and $I$, compute the $u\in(I^{[p]} :_R I)$ corresponding to the Frobenius map on $E_S$ which makes
$\Upsilon$ into an homomorphism of $S[T; f]$-modules and also find $J=\ker \Upsilon$ (cf.~\S 7 in \cite{Katzman}).
  \item Find a single parameter test element $c$ (e.g., by inspecting the Jacobian of $I$ (cf.~Chapter 2 in \cite{Huneke})).
  \item Compute $\left(cJ+I\right)^{\star u}$ (cf.~\S 5 in \cite{Katzman}).
  \item Output the $p^{e}$-weak parameter test ideal
  $\left( \left(
\left( cJ+I\right)^{\star u} \right)^{[p^e]} :_R u^{\nu_e} J
\right) S$.
\end{enumerate}

\begin{ex}
Let $\mathbb{K}$ be a field of characteristic 2, $R=\mathbb{K}[\![a,b,c,d]\!]$. Define
$$I=(a,b)R\cap (a,c)R \cap (c,d)R \cap (c+d, a^3+b d^2)R=
\left( a(c+d), bc(c+d), d(a^3+b c d) \right)R .$$
The quotient $S=R/I$ is reduced, 2-dimensional with minimal resolution
\footnote{All unjustified assertions in this and later computational examples are based on calculations carried out with \cite{Macaulay2}.}
$$
0 \rightarrow R^2 \xrightarrow[]{\left( \begin{array}{rr} bc & a^2 d\\ a & d^2\\ 0 & c+d\\ \end{array} \right)}
R^3 \xrightarrow[]{\left( \begin{array}{rrr}  a(c+d) &  bc(c+d) & d(a^3+b c d) \end{array}  \right)} R \rightarrow S \rightarrow 0
$$
which shows that $S$ is Cohen-Macaulay of type 2, hence non-Gorenstein.
The canonical module of $R/I$ is computed from this resolution as
$\Ext^2_R(R/I, R)$ which is isomorphic to the ideal $J$ generated by
$a$ and $d^2$.
The element $u\in (I^{[p]} :_R I)$ corresponding to the Frobenius map on $E_S$ is
$u=a d (c+d) (a^3 + b c d)$ and
$I_1(uR+I)=(a^2, d)R + I\neq R$, hence $S$ is not $F$-injective (cf.~Theorem 4.6 in \cite{Katzman}.)
The calculations in steps (3) and (4) in the algorithm above produce $2^i$-weak parameter test ideals $\overline{\tau}_i$
$$\overline{\tau}_0=\left( ad, ac, bd^2, a^3, c(c+d), b(c+d) \right)R, $$
$$\overline{\tau}_1=\overline{\tau}_2=\left( ad, ac, bd a^3, c(c+d), bc \right)R$$
and using Corollary \ref{Corollary: stabilization of test ideals} we deduce that $\overline{\tau}_i=\overline{\tau}_1$
for all $i\geq 1$.
\end{ex}

\section{Quasimaximal filtrations}\label{section: quasimaximal filtrations}

Throughout this section we consider a fixed $S[T;f]$-module structure of $E_S$ corresponding to a fixed $u\in (I^{[p]} : I)$,
as described in section \ref{Section: Introduction}.

As in section 4 of \cite{Lyubeznik},
for any $S[T;f]$-module $M$ we write $M_{\text{red}}$ for $M/\Nil(M)$ and
$M^*$ for the $S[T; f]$-submodule $\cap_{e\geq 0} ST^e M$ of $M$.
We note that if $M$ is Artinian as an $S$-module, there exists an $\alpha\gg 0$ such that
\begin{eqnarray*}
( M_{\text{red}} )^* & = & \frac{(\cap_{e\geq 0} ST^e M)+\Nil(M)}{\Nil(M)}\\
&=& \frac{ST^\alpha M+\Nil(M)}{\Nil(M)}\\
&\cong&\frac{ ST^\alpha M}{\Nil(M)\cap ( ST^\alpha M)}\\
&=&\frac{\cap_{e\geq 0} ST^e M}{\Nil(M)\cap (\cap_{e\geq 0} ST^e M)}\\
&=&(M^*)_{\text{red}}
\end{eqnarray*}
and
denote both of these   $M_{\text{red}}^*$.
We also recall the following:

\begin{defi}
A filtration
$0=M_0 \subset \dots \subset M_s=M$ of an $S[T;f]$-module $M$ is called  \emph{quasimaximal} if for all $1\leq i\leq s$ the modules
$(M_i/M_{i-1})_{\text{red}}^*$ are non-zero simple $S[T;f]$-modules.
\end{defi}
Artinian $S[T;f]$-modules have quasimaximal filtrations; their lengths and simple factors are invariants of the module (cf.~section 4 in \cite{Lyubeznik}).

In this section we study quasimaximal filtrations of $E_S$ and in doing so we introduce an operation on $E_S$-ideals which will be a key ingredient
for obtaining the results of the next section. We start with a description of such filtrations in general.

\begin{defi}
Let $M$ be an $S[T; f]$-module.
We define $\mathcal{A}(M)$ to be the set of all $S[T;f]$-submodules $N\subseteq M$ with the property that
$\Nil(M/N)=0$.
\end{defi}

\begin{thm}\label{Theorem: General quasi-maximal filtrations}
Let $M$ be an $S[T; f]$-module which is Artinian as an $S$-module.
Let
$0 \subseteq N_1 \subsetneq \dots \subsetneq N_s=M $
be a chain with $N_1, \dots, N_s\in \mathcal{A}(M)$ which is saturated in the sense that
for all $1\leq i\leq s-1$, there is no element in $\mathcal{A}(M)$ strictly between $N_i$ and $N_{i+1}$ and
there is no element in $\mathcal{A}(M)$ strictly contained in  $N_1$.
Then
$0 \subsetneq N_1 \subsetneq \dots \subsetneq N_s=M $
is a quasimaximal filtration of $M$ whenever $N_1\neq 0$ and
$N_1 \subsetneq \dots \subsetneq N_s=M $
is a quasimaximal filtration of $M$ whenever $N_1= 0$.
\end{thm}
\begin{proof}
Fix any $1\leq i\leq s-1$.
We have $\Nil(N_{i+1}/N_i)\subseteq \Nil(M/N_i)=0$ hence
$(N_{i+1}/N_i)^*_\text{red}=(N_{i+1}/N_i)^*$.

Pick any $S[T;f]$-submodule $A\subseteq M$ such that $N_i \subseteq A \subseteq N_{i+1}$ and let $B$ be the $S[T;f]$-submodule of
$M$ for which $\Nil(M/A)=B/A$.
We have
$$\Nil(M/B)=\Nil\left( \frac{M/A}{B/A}\right)=\Nil\left( \frac{M/A}{\Nil(M/A)}\right)=0 $$
so $B\in \mathcal{A}(M)$.
Also, the natural surjection $M/A \twoheadrightarrow M/N_{i+1}$ maps
$\Nil(M/A)=B/A$ into $\Nil(M/N_{i+1})=0$ hence $B\subseteq N_{i+1}$.
Now $N_i \subseteq B \subseteq N_{i+1}$ and the saturation of our chain implies that either
$B=N_i$ (in which case $A=N_i$)
or $B = N_{i+1}$.

We now show that $(N_{i+1}/N_i)^*=(N_{i+1}^*+N_i)/N_i$ is simple.
Pick any sub-$S[T;f]$-module $A/N_i$ of $(N_{i+1}/N_i)^*$ where $A\subseteq M$ is an $S[T;f]$-submodule
of $M$ containing $N_i$ for which $A/N_i \subseteq (N_{i+1}/N_i)^*$.
Now $(A/N_i)^* \subseteq (N_{i+1}/N_i)^*$;
if $(A/N_i)^*=(A^*+N_i)/N_i=0$, then $A^*\subseteq N_i$ and $A/N_i\subseteq \Nil(M/N_i)=0$.
Assume that $(A/N_i)^*\neq 0$ and let $B$ be as in the previous paragraph, i.e., $\Nil(M/A)=B/A$.
Since $ST^e B \subseteq A$ for all $e\gg 0$, we have $B^*=A^*$, hence
$(A/N_i)^*=(B/N_i)^*=(N_{i+1}/N_i)^*$.
Now
$(N_{i+1}/N_i)^*=(A/N_i)^*\subseteq A/N_i \subseteq (N_{i+1}/N_i)^*$
so $A/N_i = (N_{i+1}/N_i)^*$.

It remains to show that, if $N_1\neq 0$,
$$(N_1)^*_\text{red}=\left(\frac{N_1}{N_1\cap \Nil(M)}\right)^*$$
is simple. To simplify notation, write $N=N_1$.
Pick any $S[T;f]$-submodule $A$ of $M$ for which
$N\cap \Nil(M) \subseteq A \subseteq N$, and, as before, write $\Nil(M/A)=B/A$ for
an $S[T;f]$-submodule $B$ of $M$. Again we have $B\in \mathcal{A}(M)$ and $B\subseteq N$, so $B=N$.
Pick any
$(A/N\cap \Nil(M))^* \subseteq (N/N\cap \Nil(M))^*$ and
assume $(A/N\cap \Nil(M))^*\neq 0$;
again we have $A^*=B^*$ and
$(N/N \cap \Nil(M))^*=(A/N\cap \Nil(M))^*\subseteq A/N\cap \Nil(M) \subseteq (N/N\cap \Nil(M))^*$
so $A/(N \cap \Nil(M))= (N/N\cap \Nil(M))^*$.
\end{proof}

We now produce quasimaximal filtrations of $E_S$ when it is $T$-torsion free. These are described in terms of prime $E_S$-ideals.
Recall that in this case the set of $E_S$-ideals coincides with the set of $E_S$-special ideals (cf.~\S 6 in \cite{Katzman})
and that this set is finite (cf.~Theorem 3.10 in \cite{Sharp}).

\begin{cor}\label{Corollary: quasimaximal filtrations of $E_S$ in the $T$-torsion free case}
Assume that $E_S$ is $T$-torsion free and let $P_1, \dots, P_n$ be all its prime $E_S$-ideals ordered so that
$P_i \nsubseteq P_j$ for all $1\leq i < j \leq n$.
The chain
$$ 0 \subset \Ann_{E_S} P_1 \subset
\dots\subset
\Ann_{E_S} \bigcap_{j=1}^i P_j \subset\dots\subset
\Ann_{E_S} \bigcap_{j=1}^n P_j \subset E_S.
$$
is a quasimaximal filtration of $E_S$.
Therefore, the set of annihilators of the factors of any quasimaximal filtration of $E_S$ is
$\left\{ P_1, \dots, P_n \right\}$.
\end{cor}
\begin{proof}
Notice that the ordering above of $\mathcal{I}=\{P_1, \dots, P_n\}$ can always be achieved:
start with $P_{i_1}, \dots, P_{i_{n_1}}$ maximal with respect to inclusion in $\mathcal{I}$,
then list $P_{i_{n_1}}, \dots, P_{i_{n_2}}$ maximal with respect to inclusion in $\mathcal{I}\setminus \{P_{i_1}, \dots, P_{i_{n_1}}\}$, etc.

Write $A_i=\Ann_{E_S} \bigcap_{j=1}^{i} P_j$ for all $0\leq i \leq n$; notice that our ordering guarantees that these
form a strictly ascending chain.
For all $1\leq i\leq n$,  $A_{i}$ is a graded annihilator submodule of $E_S$;
since $E_S$ is $T$-torsion free, so is $E_S/A_{i}$ and $A_{i}\in \mathcal{A}(E_S)$.

Now any $S[T;f]$-submodule  between $A_{i-1}$ and $A_i$ would be a graded annihilator submodule
of the form $B=\Ann_{E_S} J $ where
$J=P_{j_1} \cap \dots \cap P_{j_m}$ with $i \leq j_1, \dots, j_m \leq n$ is a proper $E_S$-special and
$$ P_1 \cap \dots \cap P_{i-1} \cap P_i \subseteq J  \subseteq P_1 \cap \dots \cap P_{i-1} .$$
The first inclusion above shows that for all $1\leq k\leq m$, $P_{i_k}$ contains one of $P_1, \dots, P_i$ and our ordering then
shows that $i_k\leq i$.
The second inclusion above now shows that
either $J=P_1 \cap \dots \cap P_{i-1} \cap P_i$ or $P_1 \cap \dots \cap P_{i-1}$,
i.e., $B=A_{i-1}$ or $B=A_i$.
We deduce that the factors $A_i/A_{i-1}$ are simple for all $1\leq i\leq n$ and so our chain of of modules in $\mathcal{A}(E_S)$
is saturated. The result now follows from Theorem \ref{Theorem: General quasi-maximal filtrations}.
\end{proof}

The rest of this section will describe quasimaximal filtrations of $E_S$ in the presence of $T$-torsion.

\begin{prop}
For any $E_S$-ideal $J\subseteq R$, and any $e\geq 0$ we have
$$I_{e}(u^{\nu_{e}} J)  \supseteq I_{e+1}(u^{\nu_{e+1}} J) .$$
\end{prop}
\begin{proof}
First, $u^{\nu_{e+1}} J = u^{p \nu_{e}} u J \subseteq  u^{p \nu_{e}}  J^{[p]}$, so
$$I_{e+1} \left( u^{\nu_{e+1}} J\right) \subseteq I_{e+1} \left( u^{p\nu_{e}} J^{[p]}\right) .$$
Now
$I_e(u^{\nu_e}J)^{[p^{e+1}]} \supseteq \left(u^{\nu_e} J\right)^{[p]}=u^{p \nu_e} J^{[p]}$ and the minimality of
$I_{e+1} \left( u^{p \nu_e} J^{[p]} \right)$ implies that
$I_{e+1} \left( u^{p \nu_e} J^{[p]} \right) \subseteq I_e(u^{\nu_e}J)$.
\end{proof}

For any $E_S$-ideal $J$ the sequence $\{I_e(u^{\nu_e} J) \}_{e\geq 0}$ is decreasing and we can introduce the following definition.
\begin{defi}
For any $E_S$-ideal $J\subseteq R$ let
$$J^{\sharp u}=\bigcap_{e\geq 0} I_e(u^{\nu_e} J) +I.$$
\end{defi}

\medskip
Notice that $R^\sharp =\bigcap_{e\geq 0} I_e(u^{\nu_e} R) +I$ defines the submodule of nilpotent elements, i.e.,
$\Nil(E_S)=\Ann_{E_S} R^\sharp$ (cf.~Theorem 4.6 in \cite{Katzman}).

\begin{lem}
For any $E_S$-ideal $J\subseteq R$, $J^{\sharp u}$ is an $E_S$-ideal.
\end{lem}
\begin{proof}
It is enough to show that for all $e\geq 0$, $u I_e(u^{\nu_e} J) \subseteq I_{e+1}(u^{\nu_{e+1}} J)^{[p]}$.
Now
$$\left(I_{e+1}(u^{\nu_{e+1}} J)^{[p]}\right)^{[p^e]} = I_{e+1}(u^{\nu_{e+1}} J)^{[p^{e+1}]} \supseteq u^{\nu_{e+1}} J$$
so $I_{e+1}(u^{\nu_{e+1}} J)^{[p]}\supseteq I_e(u^{\nu_{e+1}} J)=I_e(u^{p^e} u^{\nu_{e}} J)$
so it is enough to show that for any $a\in R$ and any ideal $B\subseteq R$ we have
$I_e(a^{p^e} B)=a I_e(B)$.

Now $a^{p^e} B \subseteq I_e(a^{p^e} B)^{[p^e]}$ so
$$B\subseteq \left( I_e(a^{p^e} B)^{[p^e]} :_R a^{p^e} \right)= \left( I_e(a^{p^e} B) :_R a \right)^{[p^e]}$$
and so $I_e(B)\subseteq \left( I_e(a^{p^e} B) :_R a \right)$, and, therefore, $a I_e(B)\subseteq  I_e(a^{p^e} B)$.
On the other hand, $a^{p^e} B\subseteq \left(a I_e(B)\right)^{[p^e]}$, so $ I_e(a^{p^e} B) \subseteq a I_e(B)$.
\end{proof}

\begin{thm}\label{Theorem: sharp}
Let $M$ be an $S[T;f]$-submodule of $E_S$ and write $M=\Ann_{E_S} J$ for an $E_S$-ideal $J$.
Then $\Nil(E_S/M)=\Ann_{E_S} J^{\sharp u}/M$.
\end{thm}
\begin{proof}
Let $N_e$ be the $S[T;f]$-submodule of $E_S/M$ consisting of all elements killed by $T^e$.

An application of the functor $\Delta^e$ (cf.~section 4 in \cite{Katzman}) to the short exact sequence
$0 \rightarrow M \rightarrow E_S \rightarrow E_S/M \rightarrow 0$
yields the following short exact sequence in  $\mathcal{D}^e$
\begin{equation}\label{CD1}
\xymatrix{
0 \ar@{>}[r]^{} &
\displaystyle\frac{J}{I} \ar@{>}[r]^{} \ar@{>}[d]^{u^{\nu_e}} &
\displaystyle\frac{R}{I} \ar@{->}[r]^{} \ar@{>}[d]^{u^{\nu_e}} &
\displaystyle\frac{R}{J} \ar@{>}[r]^{} \ar@{>}[d]^{u^{\nu_e}}  &
0 \\
0 \ar@{>}[r]^{} &
\displaystyle \displaystyle\frac{J^{[p^e]}}{I^{[p^e]}}  \ar@{>}[r]^{} &
\displaystyle \displaystyle\frac{R}{I^{[p^e]}}        \ar@{>}[r]^{} &
\displaystyle \displaystyle\frac{R}{J^{[p^e]}}        \ar@{>}[r]^{} &
0 \\
}.
\end{equation}
Write $J_e=I_e(u^{\nu_e} J)+I$ and
consider the following exact sequence in $\mathcal{D}^e$
\begin{equation}\label{CD2}
\xymatrix{
\displaystyle\frac{J}{I} \ar@{>}[r]^{} \ar@{>}[d]^{u^{\nu_e}} &
\displaystyle\frac{J}{J_e} \ar@{->}[r]^{} \ar@{>}[d]^{u^{\nu_e}} &
0 \\
\displaystyle \displaystyle\frac{J^{[p^e]}}{I^{[p^e]}}  \ar@{>}[r]^{} &
\displaystyle \displaystyle\frac{J^{[p^e]}}{J_e^{[p^e]}}        \ar@{>}[r]^{} &
0 \\
}.
\end{equation}
Write $N_e^\prime=\Psi^e\left( \frac{J}{J_e} \xrightarrow[]{u^{\nu_e}} \frac{J^{[p^e]}}{{J_e^{[p^e]}}} \right)$
and note that it is an $S[T;f]$-submodule of $E_S/M$.
The definition of $J_e$ implies that the rightmost map in (\ref{CD2}) is zero, hence
$T^e N_e^\prime=0$ so $ N_e^\prime\subseteq  N_e$.
On the other hand, an application of $\Delta^e$ to the exact sequence $0 \rightarrow N_e \rightarrow E_S/M$
yields an exact sequence in $\mathcal{D}^e$
\begin{equation}\label{CD3}
\xymatrix{
\displaystyle\frac{J}{I} \ar@{>}[r]^{} \ar@{>}[d]^{u^{\nu_e}} &
\displaystyle\frac{J}{L_e} \ar@{->}[r]^{} \ar@{>}[d]^{u^{\nu_e}} &
0 \\
\displaystyle \displaystyle\frac{J^{[p^e]}}{I^{[p^e]}}  \ar@{>}[r]^{} &
\displaystyle \displaystyle\frac{J^{[p^e]}}{{L_e}^{[p^e]}}        \ar@{>}[r]^{} &
0 \\
}.
\end{equation}
for some $E_S$-ideal $L_e$ such that $I\subseteq L_e \subseteq J$ and for which $u^{\nu_e} J\subseteq L_e^{[p^e]}$. Now the minimality of
$J_e=I_e(u^{\nu_e} J)+I$ implies that $I_e(u^{\nu_e} J)\subseteq L_e$ and hence
$$N_e=\left(\frac{J}{L_e}\right)^\vee =\frac{\Ann_E L_e}{M}\subseteq \frac{\Ann_E J_e}{M}=N_e^\prime$$
and we deduce that $N_e=N_e^\prime$.

We now conclude the proof by observing that
$$
\Nil(E_S/M)=\bigcup_{e\geq 0} N_e=\bigcup_{e\geq 0} N_e^\prime= \bigcup_{e\geq 0}  \frac{\Ann_{E_S} J_e}{M}=
\frac{ \Ann_{E_S} \bigcap_{e\geq 0}  J_e}{M}=\frac{\Ann_{E_S} J^{\sharp u}}{M} .$$
\end{proof}

\begin{cor}
For any $E_S$-ideal $J\subseteq R$,
$E_S/\Ann_{E_S} J^{\sharp u}$ is $T$-torsion free and
$\left( J^{\sharp u}\right)^{\sharp u}= J^{\sharp u}$.
\end{cor}

\begin{defi}
We define
$$\mathcal{I}^\sharp =\left\{ J^{\sharp u} \,|\, J\subseteq R  \text{ is an } E_S \text{ ideal} \right\} $$
and call a chain $J^{\sharp u}_0 \subset J^{\sharp u}_1 \subset \dots \subset J^{\sharp u}_\ell$ of ideals in $\mathcal{I}^\sharp$
\emph{$\sharp$-saturated} if it cannot be refined by adding an ideal in $\mathcal{I}^\sharp$.
\end{defi}

\begin{thm}\hfil
Let
$I=0^{\sharp u}=J^{\sharp u}_0 \subset J^{\sharp u}_1 \subset \dots \subset J^{\sharp u}_\ell=R^{\sharp u}$ be a
$\sharp$-saturated chain.
Then
$$0\subset J^{\sharp u}_{\ell-1} \subset \dots \subset \Ann_{E_S} J^{\sharp u}_1 \subset \Ann_{E_S} J^{\sharp u}_0=E_S$$
is a quasi-maximal filtration of $E_S$.
\end{thm}
\begin{proof}
Notice that $\mathcal{A}(E_S)= \{ \Ann_{E_S} J \,|\, J\in \mathcal{I}^\sharp \}$.
Any finite strictly ascending chain in $\mathcal{A}(E_S)$ can be refined to saturated chain and all these have the same length,
namely the quasilength of $E_S$ (cf.~Theorem 4.6 in \cite{Lyubeznik}).
So finite saturated chains as in the statement of the theorem do exist and now the theorem follows
from Theorem \ref{Theorem: General quasi-maximal filtrations}.
\end{proof}

The ideals $J^{\sharp u}$ will play a central role in calculating tight closure in $E_S$ as described in the following section.

\section{Tight closure in $E_S$}\label{Tight closure in $E_S$}
In this section we give an explicit
\footnote{An ``explicit'' description is interpreted in this section as one given as the annihilator in $E_S$ of an explicitly given ideal.
This is justified in view of Matlis duality and in view of the fact that in the context of the theory of tight closure one is usually interested
in the annihilators of submodules of $E_S$ rather than in the submodules themselves.}
description of the tight closure of certain submodules of $E_S$, including $0^*_{E_S}$,
which holds whenever the $S$-algebra $\mathcal{F}(E_S)$ is generated by one element.
The class of complete local rings $S$ with this property includes those which are quasi-Gorenstein, but it is strictly larger than this
as is illustrated by the example at the end of this section.

We shall henceforth use the natural isomorphism $\mathcal{F}^e(M)\cong \Hom_S \left( ST^e \otimes_S M , M\right)$
which maps a $\phi\in  \mathcal{F}^e(M)$ to the $S$-linear map
$\widetilde{\phi} : ST^e \otimes M \rightarrow M$ determined by $\widetilde{\phi}(s\otimes m)=s\phi(m)$
(cf.~ section 3 in \cite{Lyubeznik-Smith}). Conversely, the element $\widetilde{\phi} : ST^e \otimes M \rightarrow M$
corresponds under this isomorphism to the map  $\phi\in\mathcal{F}^e(M)$ given by
$\phi(m)=\widetilde{\phi}(1\otimes m)$.
We shall henceforth identify these two $S$-modules using this notation.

\bigskip
We can think of the tight closure of ideals $L\subseteq S$ as the set of all elements $s\in S$ such that
for some $c\in S$ not in any minimal prime we have
$c \phi(s) \in S \phi(L)$ for all $e\gg 0$ and all $\phi\in \mathcal{F}^e(S)$.
This is because for each $e\geq 0$, $\mathcal{F}^e(S)$ is generated by the $e^\text{th}$ iterated Frobenius map on $S$.
Our first aim is to show that this also yields the tight closure of submodules of $E_S$, and to do so we  shall need
weak test elements for testing tight closure in this setup.
\begin{defi}
Let $M$ be an $S$-module and let $N\subseteq M$ be an $S$-submodule.
We call $c\in S$ not in any minimal prime a $p^\eta$-weak test element for the pair $(N,M)$
if $a\in N^*_M$ if an only if $c\otimes a\in S T^e\otimes M $ is in the image of
$S T^e\otimes N$ in $S T^e\otimes M$ for all $e\geq \eta$.
Henceforth $N^{[p^e]}_M$ (or just $N^{[p^e]}$ when it will not lead to confusion) will denote the image of $S T^e\otimes N$ in $S T^e\otimes M$.
\end{defi}
These test elements mentioned in the definition above are known to exist when $S$ is $F$-pure (cf.~ section 3 in \cite{Sharp2}),
and I believe they exist in much wider generality.

\begin{prop}\label{Proposition: tight closure as intersection of special tight closure}
Let $N$ be any $S$-submodule of $E_S$, let $c\in S$ and fix an $a\in M$.
For all $e\geq 0$,
$c\otimes a\in ST^e \otimes E_S$ is in  $N^{[p^e]}$ if and only if
for all $\phi\in \mathcal{F}^e(E_S)$ we have $c \phi(a) \in S\phi(N)$.

Consequently, if $c$ is a $p^\eta$-weak test element for the pair $(N,E_S)$ then
$a\in N^*_{E_S}$ if and only if for all $e\geq \eta$ and all $\phi\in \mathcal{F}^e(E_S)$ we have $c \phi(a) \in S\phi(N)$.
\end{prop}
\begin{proof}
Note that for all $\phi\in \mathcal{F}^e(E_S)$ we have $\widetilde{\phi}(N^{[p^e]})=S \phi(N)$.

Assume first that  $c\otimes a\in N^{[p^e]}$.
Now for all  $\phi\in \mathcal{F}^e(E_S)$ we have
$$c\phi(a)=\widetilde{\phi}(c\otimes a)\subseteq \widetilde{\phi}\left( N^{[p^e]}\right)= S\phi(N) .$$

Assume now that $c \phi(a) \in S\phi(N)$ for all
$\phi\in \mathcal{F}^e(E_S)$.
Let $M\subseteq ST^e\otimes_S E_S$ be the $S$-submodule generated by  $N^{[p^e]}$ and $c\otimes a$.
The inclusion above yields a surjection $\Hom_S( ST^e\otimes_S E_S, E_S) \twoheadrightarrow \Hom_S(M, E_S)$;
we now recall that $\mathcal{F}^e(E_S) =\Hom_S(ST^e\otimes_S E_S, E_S)$ and deduce that for all
$\widetilde{\phi}\in \Hom_S( ST^e\otimes_S E_S, E_S)$ we have
$\widetilde{\phi}(M) = \widetilde{\phi}(N^{[p^e]}+S(c\otimes a))= S\phi(N)=\widetilde{\phi}(N^{[p^e]})$.

If $c\otimes a\notin N^{[p^e]}$ we can find a non-zero
$\overline{\psi} \in \Hom_S ( M/ N^{[p^e]} , E_S)$.
The short exact sequence
$$ 0 \rightarrow \Hom_S ( \frac{M}{ N^{[p^e]}} , E_S) \rightarrow
\Hom_S ( M , E_S) \rightarrow
\Hom_S ( N^{[p^e]} , E_S) \rightarrow 0$$
enables us to identify $\overline{\psi}$ with a
non-zero $\psi \in \Hom_S ( M , E_S)$ for which $\psi(N^{[p^e]})=0$.
Since $E_S$ is injective, we can extend $\psi$ to an element $\widetilde{\psi} \in \Hom_S (ST^e\otimes_S E_S , E_S)$.
Now $\widetilde{\psi}(c\otimes m)\neq 0$, otherwise $\overline{\psi}=0$,
and hence $\widetilde{\psi}(M)\neq 0$ and so is not equal to $\widetilde{\psi}(N^{[p^e]})=0$,  contradicting the conclusion of the previous paragraph.

The final conclusion follows directly from the definition of  weak test elements.
\end{proof}

\bigskip
The proposition above gives a method for translating the calculation of the tight closure of an $S$-submodule  $N\subseteq E_S$ to a calculation involving
ideals of $R$ as follows. Assume $c\in S$ be a $p^\eta$-weak test element for the pair $(N,E_S)$.
Fix an $e\geq \eta$,  $\phi\in \mathcal{F}^e(E_S)$ and the corresponding $S[\Theta; f^e]$-module structure on $E_S$
corresponding to $v\in(I^{[p^e]} : I)$.
Define $N_\phi=\{ m\in E_S \,|\, c\Theta m \in S\Theta N\}$ and write $N_\phi=\Ann_{E_S} L_\phi$ for some ideal $L_\phi\subseteq R$.
Notice that $N_\phi$ is the largest submodule of $E_S$ with the property
$c S\Theta N_\phi \subseteq S\Theta N$, i.e.,
$c \Ann_{E_S} ( 0 :_R S\Theta N_\phi) \subseteq \Ann_{E_S}  ( 0 :_R S\Theta N)$
which, using the  $S[\Theta; f^e]$-module analogue of Theorem \ref{Theorem: graded annihilators}(b), translates to
$c \Ann_{E_S} (L_\phi^{[p^e]} : v)  \subseteq \Ann_{E_S} (J^{[p^e]} : v)$,
or, equivalently,
$(L_\phi^{[p^e]} : v)  \supseteq c (J^{[p^e]} : v)$, i.e.,
$L_\phi^{[p^e]}  \supseteq c v (J^{[p^e]} : v)$.
We deduce that $L_\phi$ is the minimal ideal $L\subseteq R$ containing $I$
for which $L^{[p^e]}  \supseteq c v (J^{[p^e]} : v)$, i.e.,
$L_\phi= I_e \left( c v (J^{[p^e]} : v) \right)+ I $.
We can now express $N^*_{E_S}$ as the annihilator in $E_S$ of the sum of all these ideals $L_\phi$.

In some simple cases this gives directly a fairly explicit expression for the tight closure on $N$.
For example, if $I$ is generated by a regular sequence $g_1, \dots, g_m$ and $N=0$, then
for all $e\geq 0$ we have $(I^{[p^e]} : I)=g^{p^e-1}+I^{[p^e]} $ where $g=g_1 \cdot \ldots \cdot g_m$ and,
if $c$ is a test element for $(0,E_S)$, then
$$0^*_{E_S}=\Ann_{E_S} \sum_{e\geq 0} I_e(c g^{p^e-1}) + I.$$

The rest of this section applies Proposition \ref{Proposition: tight closure as intersection of special tight closure} under
additional hypothesis to produce
explicit expressions for $N^*_{E_S}$: we shall first restrict our attention to $\mathcal{F}(E_S)$ submodules $N\subseteq E_S$ (which includes the interesting case
where $N=0$) and later
we shall impose the additional condition that the $S$-algebra $\mathcal{F}(E_S)$ is generated by one element.

\begin{prop}\label{Proposition: special tight closure}
Fix any $S[\Theta;f^\eta]$-module structure on $E_S$.
Let $c\in R$ and let $Z=\Ann_{E_S} J$ be an $S[\Theta;f^\eta]$-submodule, where $I\subseteq J \subseteq R$ is an ideal.
Let $Y=\Ann_{E_S} L$ be the largest $S[\Theta;f^\eta]$-submodule of $E_S$ contained in $\Ann_{E_S} c J$
where $cJ \subseteq L \subseteq R$ is an ideal.
Choose a positive integer $j_0$  such that
$\Theta^{j_0} \Nil(E_S/Y)=0$.
Write
$$M=\left\{ m\in E_S \,|\, c \Theta^j m \in Z\text{ for all } j\geq j_0\right\} .$$
Then $M$ is the preimage in $E_S$ of $\Nil \left( E_S / Y\right)$
\end{prop}
\begin{proof}
Clearly, if $m+ Y \in\Nil \left( E_S / Y\right)$ then for all
$j\geq j_0$ we have $L \Theta^j m=0$ and since $L\supseteq cJ$ we also have $c \Theta^j m \in \Ann_{E_S} J=Z$.

Notice that $M$ is an $S[\Theta;f^\eta]$-submodule of $E_S$.
Write $A=(0 :_R S\Theta^{j_0} M)$;
Theorem \ref{Theorem: graded annihilators}(d) shows that $\Ann_{E_S} A$ is an $S[\Theta;f^\eta]$-submodule of $E_S$.
Furthermore,
$c S \Theta^{j_0} M \subseteq \Ann_{E_S} J$, i.e., $c J \Theta^{j_0} M=0$ and hence $cJ \subseteq A$
implying $L\subseteq A$.
Since
$S \Theta^{{j_0}} M\subseteq  \Ann_{E_S} A$ we have
$m+ \Ann_{E_S} A \in \Nil\left(E_S/\Ann_{E_S} A\right)$
for all $m\in M$;
as $L\subseteq A$
we also have
$m + \Ann_{E_S} L\in \Nil\left(E_S/\Ann_{E_S} L\right)$.
\end{proof}

Our next goal is produce an explicit method of calculating tight closure in $E_S$. The following introduces the main
tool.

\begin{defi}
Let $e\geq 0$, fix any $u\in(I^{[p^e]} : I)$ and let  $J\subseteq R$ be any ideal containing $I$.

We write  $J^{\star^e u}$ for the smallest ideal  $L$ containing $J$
for which $u L \subseteq L^{[p^e]}$ (see section 5 in \cite{Katzman} for a construction of this ideal).

Endow $E_S$ with the structure of an  $S[\Theta; f^e]$-module corresponding to $u$, and let
$M$ be an $S$-submodule of $E_S$.
We define $M^{\star^e}$ to be the largest $S[\Theta; f^e]$-submodule of $E_S$ contained in $M$.

Note that if $M=\Ann_{E_S} J$, $M^{\star^e}=\Ann_{E_S} J^{\star^e u}$.
\end{defi}

\begin{thm}\label{Theorem: Formula for tight closure}
Suppose that the $S$-algebra $\mathcal{F}(E_S)$ is generated by one element corresponding to $u\in (I^{[p]} : I)$.
Let $N$ be a $S[T; f]$-submodule of $E_S$ and let $Z=\Ann_{E_S} J$ be the stable value of
the descending chain $\{ ST^j N \}_{j\geq 0}$.
Assume further that the image of $c\in R$ in $S$ is a weak $p^{\eta}$-test element
for the pair $(N,E_S)$ and that $\eta$ was chosen so large that $Z=ST^\eta N$.
We have
$$N^*_{E_S}= \Ann_{E_S} \left((cJ+I)^{\star^\eta u}\right)^{\sharp u}  .$$
\end{thm}

\begin{proof}
Fix the $S[T; f]$-module structure on $E_S$ corresponding to $u$.
In view of Proposition \ref{Proposition: tight closure as intersection of special tight closure}
and of the fact that for all $e\geq \eta$, $\mathcal{F}^e(E_S)=ST^e$,
$$N^*_{E_S}= \cap_{e\geq \eta} \{ m\in E_S \,|\, c T^e m\in S T^e N \}$$
for all $\eta\geq \eta_0$.

Notice that, if for some $a\in E_S$ and positive integer $j$
the element $c\otimes a\in ST^{j \eta}\otimes E_S$ is in $N^{[p^{j \eta}]}$, then
after tensoring on the left with $ST^k$ for $1\leq k\leq \eta -1$ and using the isomorphism
$ST^k \otimes ST^{j \eta}\cong ST^{k+j \eta}$, we obtain
$c^{p^k} \otimes a\in ST^{j \eta+k} \otimes E_S$ is in $N^{[p^{j \eta+k}]}$ for all $1\leq k\leq \eta -1$ and hence
$c^{p^{\eta-1}} \otimes a\in ST^{j \eta+k} \otimes E_S$ is in $N^{[p^{j \eta}+k]}$ for all $1\leq k\leq \eta -1$.
For any positive integer $j_0$, we may replace $c$ with $c^{p^{\eta-1}}$ as a $p^{j_0 \eta}$ weak test element and
deduce that
$$a\in N^*_{E_S}= \cap_{j\geq j_0} \{ m\in E_S \,|\, c T^{j \eta} m\in S T^{j \eta} N \}. $$

Write $\Theta=T^\eta$ and let $L=(cJ+I)^{\star^\eta u}$.
Note that $\Ann_{E_S} L$ is the largest $S[\Theta; f^\eta]$-submodule of $E_S$ contained in $\Ann_{E_S} cJ$.
Pick any positive integer $j_0$ such that $\Theta^{j_0} \Nil(E_S/\Ann_{E_S} L)=0$.

An application of Proposition \ref{Proposition: special tight closure}
shows that
$$N^*_{E_S}= \cap_{j\geq j_0} \{ m\in E_S \,|\, c \Theta^j m \in  Z \}$$
is the pre-image of $\Nil(E_S/\Ann_{E_S} L)$ in $E_S$,
and this is precisely
$\Ann_{E_S} L^{\sharp u}$.
\end{proof}

The theorem above can be easily translated into an algorithm.
However, the practicality of this algorithm will be limited by the fact that one of its inputs
is an element $c\in R$ whose image in $S$ is a weak $p^{\eta}$-test element 
for the pair $(N,E_S)$.
Although there are currently no published methods for producing such a test element,
there are preliminary results by Mel Hochster showing that these can be chosen to be suitable powers of elements in
the singular locus of $S$, just as with regular test elements.

\bigskip
One instance when the $S$-algebra $\mathcal{F}(E_S)$ is generated by one element is when
$S$ is Gorenstein, or more generally, when  $S$ is quasi-Gorenstein (i.e., $E_S\cong \HH^{\dim S}_{\mathfrak{m} S} (S)$)
and satisfies Serre's $S_2$ condition. This is the content of Example 3.6 in \cite{Lyubeznik-Smith}.
However, the class of quotients $S$ of $R$ for which the $S$-algebra $\mathcal{F}(E_S)$ is generated by one element is
strictly larger than this.

\begin{ex}
Consider the power series ring $R=\mathbb{K}[\![ a,b,c ]\!]$, where $\mathbb{K}$ is a field of prime characteristic $p$,
its ideal  $I=(ab-bc,bc-b^2, ac-bc)R=(a,b)R \cap (c,b)R \cap (a-c, b-c)R$ and the one dimensional quotient $S=R/I$.
We have a minimal resolution
$$
0 \rightarrow R^2 \xrightarrow[]{\left( \begin{array}{rr} b & c\\ -a & -c\\ -b & -b\\ \end{array} \right)}
R^3 \xrightarrow[]{\left( \begin{array}{rrr} ab-bc & bc-b^2 & ac-bc \end{array}  \right)} R \rightarrow S \rightarrow 0
$$
which shows that $S$ is Cohen-Macaulay of type 2, hence non-Gorenstein and not quasi-Gorenstein.
For all primes $p\geq 5$,
$$ b^{p-1} (b-c)^{p-1} (a-b)^{p-1} \in (I^{[p]} : I) $$
and a calculation with Macaulay2  shows that this element
generates the $S$-module $(I^{[p]} : I)/I^{[p]}$ for all $5\leq p \leq 97$.

Fix now $p=5$. We compute $I_1(uR+I)=R$, so $S=R/I$ is $F$-injective.
We choose, using the notation of Theorem \ref{Theorem: Formula for tight closure}, $N=0$, hence $J=R$, and we \emph{assume} that
the test element
$c_0=( a^2+2ab+2b^2-2ac-bc-c^2 )^3$ can be used as a test element for the pair $(0, E_S)$.
We calculate
$L=(c_0R+I)^{\star^1 u}=(a, b, c)R$ and  $I_1(u L)+I=I_2(u^{1+2} L)+I=(a, b, c)R$ and so
we can give $0^*_{E_S}$ explicitly as the annihilator in $E_S$ of
$((a, b, c)R)^{\sharp u}=(a, b, c)R$.
\end{ex}

\begin{cor}\label{Corollary: The test ideal of quasi-Gorenstein rings}
Assume that $S$ is  equidimensional and quasi-Gorenstein and that it satisfies Serre's $S_2$ condition.
Let $c\in R$ be such that its image in $S$ is a test element for the pair $(0,E_S)$.
The test ideal of $S$ is
$$\Ann_S 0^*_{E_S}= \left((Rc+I)^{\star u}\right)^{\sharp u}S  .$$
\end{cor}
\begin{proof}
The fact that $S$ is quasi-Gorenstein and equidimensional implies that
the finitistic tight closure of $0$ in $E_S$
coincides with $0^{*}_{E_S}$ and hence the test ideal of $S$ is $\Ann_S 0^*_{E_S}$
(cf.~section 8 in \cite{Hochster-Huneke-0} and Proposition 3.3 in \cite{Smith}).
The fact that $S$ satisfies Serre's $S_2$ condition implies that the $S$-algebra
$\mathcal{F}(E_S)$ is generated by one element.
Now the result follows from Theorem \ref{Theorem: Formula for tight closure} with $J=R$.
\end{proof}

\affiliationone{
Mordechai Katzman\\
Department of Pure Mathematics\\
University of Sheffield\\
Hicks Building\\
Sheffield S3 7RH\\
United Kingdom}
\email{M.Katzman@sheffield.ac.uk}

\end{document}